\documentclass[11pt]{article}
\usepackage{geometry}                % See geometry.pdf to learn the layout options. There are lots.
\usepackage{graphicx}
\usepackage{amssymb}
\usepackage{epstopdf}
\usepackage{hyperref}
\usepackage[T1]{fontenc}
\usepackage[ngerman, english]{babel}
\usepackage{mathtools}
\usepackage{mathabx}
\usepackage{bbm}
\usepackage{mathrsfs}
\usepackage{amsmath,amscd}
\usepackage{tikz}
\usepackage{tikz-cd}
\usepackage{babel}
\usetikzlibrary{arrows, matrix}
\usetikzlibrary{cd}
\usepackage[arrow, matrix, curve]{xy}

\newcommand{\qed}{\hfill \ensuremath{\Box}}
\newenvironment{proof}{\vspace{1ex}\noindent{\it Proof.}\hspace{0.5em}}
	{\hfill\qed\vspace{1ex}}

\newtheorem{theorem}{Theorem}[section]
\newtheorem{lemma}[theorem]{Lemma}
\newtheorem{proposition}[theorem]{Proposition}
\newtheorem{corollary}[theorem]{Corollary}
\newtheorem{conjecture}[theorem]{Conjecture}
\newtheorem{definition}[theorem]{Definition}

\DeclareMathOperator{\R}{\mathrm{R}}
\DeclareMathOperator{\Z}{\mathbf{Z}}
\DeclareMathOperator{\F}{\mathbf{F}}

\DeclareMathOperator{\Ps}{\mathbf{P}}

\DeclareMathOperator{\rmd}{\mathrm{d}\!}

\DeclareMathOperator{\N}{\mathbf{N}}

\DeclareMathOperator{\im}{\mathrm{im}}
\DeclareMathOperator{\Spec}{\operatorname{Spec}}

\DeclareMathOperator{\Hom}{\operatorname{Hom}}

\DeclareMathOperator{\Mor}{\operatorname{Mor}}

\DeclareMathOperator{\Lie}{\mathrm{Lie}}

\DeclareMathOperator{\Ext}{\operatorname{Ext}}

\DeclareMathOperator{\Sh}{\mathrm{Sh}}

\DeclareMathOperator{\Id}{\mathrm{Id}}
\DeclareMathOperator{\Og}{\mathcal{O}}

\DeclareMathOperator{\Pic}{\mathrm{Pic}}

\DeclareMathOperator{\rk}{\mathrm{rk}}

\DeclareMathOperator{\et}{\acute{\mathrm{e}}{\mathrm{t}}}
\DeclareMathOperator{\Gm}{\mathbf{G}_m}
\DeclareMathOperator{\Ga}{\mathbf{G}_a}

\DeclareMathOperator{\fppf}{\mathrm{fppf}}
\DeclareMathOperator{\sHom}{\mathscr{H}\!\mathit{om}}
\DeclareMathOperator{\sExt}{\mathscr{E}\!\mathit{xt}}

\DeclareMathOperator{\Res}{\mathrm{Res}}

\DeclareMathOperator{\sep}{{^\mathrm{sep}}}

\DeclareMathOperator{\RP}{{\mathrm{RP}}}
\DeclareMathOperator{\RPS}{{\mathrm{RPS}}}
\DeclareMathOperator{\Ab}{{\mathrm{Ab}}}
\title{On Bosch-Lütkebohmert-Raynaud's Conjecture I}
\author{Otto Overkamp}
\date{}
\begin{document}
\maketitle
{\abstract{Let $G$ be a smooth commutative algebraic group over the field of rational functions of an excellent Dedekind scheme $S$ of equal characteristic $p>0.$ A Néron lft-model of $G$ is a smooth separated model $\mathscr{G} \to S$ of $G$ satisfying a universal property. Predicting whether a given $G$ admits such a model is a very delicate (and, in general, open) question if $S$ has infinitely many closed points, which is the subject of \it Conjecture I \rm due to Bosch-Lütkebohmert-Raynaud. This conjecture was recently proven by T. Suzuki and the author if the residue fields of $S$ at closed points are perfect, but refuted in general. The aim of the present paper is two-fold: firstly, we give a new construction of counterexamples which is more general and provides a conceptual explanation for the only counterexamples known previously, as well as providing many new counterexamples. Secondly, we shall give a new and elementary proof of Conjecture I in the case of perfect residue fields. Both parts make use of the concept of weakly permawound unipotent groups recently introduced by Rosengarten.}}
\tableofcontents
\section{Introduction}
\subsection{Background}
Let $S$ be an excellent Dedekind scheme over $\F_p$ for some prime number $p.$ Let $K$ be the field of rational functions of $S$ and let $G$ be a smooth commutative algebraic group\footnote{By an \it algebraic group \rm we shall always mean a group scheme of finite type over a field. All group schemes appearing in this article will be commutative.} over $K.$ A \it Néron lft-model \rm \cite[Chapter 10.1, Definition 1]{BLR} of $G$ over $S$ is a smooth separated group scheme $\mathscr{G} \to S$ whose generic fibre is $G$ and which satisfies the \it Néron mapping property: \rm for each smooth morphism $T\to S$ and for each $K$-map $\phi\colon T\times_S\Spec K \to G$, there is a unique map $T\to \mathscr{G}$ over $S$ extending $\phi.$ We shall call $\mathscr{G}$ a \it Néron model \rm of $G$ if it is moreover of finite presentation (or, equivalently, quasi-compact) over $S$ \cite[Chapter 1.2, Definition 1]{BLR}. The role played by Néron (lft-)models of algebraic groups is hard to overestimate; for example, the notions of \it good reduction, semiabelian reduction, \rm and \it bad reduction \rm of Abelian varieties are defined in terms of their Néron models, and the Néron-Ogg-Shafarevich criterion, which relates the reduction behaviour of an Abelian variety to the Galois representation on its Tate module, follows almost immediately from the existence of Néron models for Abelian varieties \cite[Chapter 7.4, Theorems 5 and 6]{BLR}. 

Given a smooth commutative algebraic group $G$ over $K,$ it is therefore a fundamental question whether a Néron (lft-)model of $G$ exists. This is certainly not always the case, since the presence of a closed subgroup of $G$ isomorphic to $\Ga$ precludes the existence of a Néron lft-model of $G$ \cite[Chapter 10.1, Proposition 8]{BLR}. If $S$ is \it local, \rm
i. e. if $S$ is the spectrum of an excellent discrete valuation ring, the situation is completely understood: a Néron lft-model of $G$ exists \it if and \rm only if $G$ contains no closed subgroups isomorphic to $\Ga,$ and the Néron lft-model is quasi-compact over $S$ if and only if, moreover, $G$ contains no torus split by an unramified extension of $K$ \cite[Chapter 10.2, Theorems 1 and 2]{BLR}.

However, if $S$ is \it global, \rm i. e. if $S$ has infinitely many closed points, the question is considerably more delicate. The first indication for this was an example due to Oesterlé \cite[Chapter 10.1, Example 11]{BLR}, which shows that, if $S$ is global and $K'$ is a non-trivial finite purely inseparable extension of $K,$ the algebraic group $G=\Res_{K'/K}\Gm / \Gm$ admits a Néron lft-model $\mathscr{G}\to S$ which is not quasi-compact, but such that the Néron lft-models $\mathscr{G}\times_S \Spec \Og_{S,s}$ of $G$ over $\Og_{S,s}$ are quasi-compact for each closed point $s$ of $S.$ This is explained by the fact that the groups of connected components of $\mathscr{G}$ are finite but non-trivial for all such closed points, and shows that, in general, a \it local-to-global principle \rm for the existence of Néron (lft-)models cannot be expected. The group just considered is clearly unirational\footnote{An algebraic group $G$ over $K$ is \it unirational \rm if its underlying scheme is, i. e. if there exists a scheme-theoretically dominant morphism $U\to G$ with $U\subseteq \mathbf{A}^n_K$ open for some $n.$ Such groups are always smooth, connected, and affine.} over $K.$ These observations led Bosch-Lütkebohmert-Raynaud to make the following 
\begin{conjecture}\rm (Cf. \cite[Chapter 10]{BLR}) \it \label{BLRconj}
Let $G$ be a smooth commutative algebraic group over $K.$\\
(I) If $G$ contains no subgroup isomorphic to $\Ga,$ then $G$ admits a Néron lft-model over $S.$ \\
(II) Moreover, if $G$ contains no non-trivial unirational closed subgroup, then $G$ admits a Néron model over $S.$
\end{conjecture}

Both conjectures are known to be true if $G = \Pic^0_{C/K}$ for some proper geometrically reduced curve $C$ over $K$ \cite{OvGR}. However, for more general $G,$ Conjecture I above usually fails. More precisely, T. Suzuki and the author recently established the following

\begin{theorem} \rm (Cf. \cite{OS}) \it  
(i) Suppose that the residue fields of $S$ are perfect. Then Conjecture I holds for $S.$\\
(ii) Assume $[K:K^p]<\infty$ and that for some (equivalently, every) closed point $s\in S,$ the residue field $\kappa(s)$ is imperfect. Then there exists a smooth connected commutative algebraic group $G$ over $K$ with $\Hom_K(\Ga,G)=0$ which does not admit a Néron lft-model over $S.$ 
\end{theorem}
The purpose of the present article is to give a new proof of this result (see Theorems \ref{Nétnonexthm} and \ref{ExistenceThm}). The main innovations are the following:
\begin{itemize}
\item Our construction of counterexamples in the case of imperfect residue fields is new. While in \cite{OS}, one counterexample was given for each eligible $S$ by means of explicit equations, we show that they naturally come in infinite families and are ubiquitous; if $[K:K^p] > p^2,$ this includes infinite families of new counterexamples. The new construction relies on a cohomological purity theorem of Shiho \cite{Shi} and shows that the obstruction to the existence of Néron lft-models lives in certain local cohomology sheaves.
\item Our new proof of Conjecture I in the case of perfect residue fields is elementary and entirely bypasses the advanced machinery of duality on relatively perfect sites used in \cite{OS}. Instead, we rely on recent results due to Z. Rosengarten \cite{Ros}, which are highly non-trivial but also derived in an elementary way. 
\end{itemize}

Conjecture II is known to follow from resolution of singularities in positive characteristic \cite[Chapter 10.3, Theorem 5 (a)]{BLR}, and is also known if the residue fields of $S$ are perfect \cite{OS}. This conjecture implies, for example, that pseudo-Abelian varieties \cite[Definition 0.1, Theorem 2.1]{Totaro} over $K$ admit Néron models over a global base as soon as $[K:K^p]< \infty.$ We shall give an elementary proof of this fact (which is also implicit but hidden in \cite{OS}).
\subsection{Notation}
Throughout this article, a \it group scheme \rm will mean a group object in the category of schemes over some base $S,$ and an \it algebraic group \rm will mean a group scheme of finite type over a field. All group schemes considered in this article will be commutative. \\
\\
$\mathbf{Acknowledgement}.$ The author is grateful to Professor K. Rülling for answering his questions on de Rham-Witt complexes and to the referee for very insightful comments on a previous version. This research was undertaken in the framework of the research training group \it GRK 2240: Algebro-Geometric Methods in Algebra, Arithmetic and Topology. \rm
\section{Relative perfection and duality}

\subsection{Relative perfection}
In this paragraph we shall describe the theory of relatively perfect schemes, as well as the duality theory of relatively perfect commutative unipotent groups over fields, developed by Kato and Suzuki \cite{Kato, Suz}, to the extent necessary in this article. Let $p$ be a prime number and let $S$ be a regular excellent scheme over $\F_p.$ Assume moreover that $\Omega^1_{S}:=\Omega^1_{S/\F_p}$ is locally free of some constant finite rank $r;$ this means in particular that the absolute Frobenius $F_S \colon S\to S$ is locally free of rank $p^r$ \cite[Tags 07P2, 07QV, 0EC0]{Stacks}. For a scheme $X\to S,$ we shall write $X^{(p^n)}$ for the fibre product of $X\to S$ along the $n$-fold Frobenius $F^{n}_{S} \colon S \to S$ and $X^{(1/p^n)}$ for the Weil restriction of $X\to S$ along $F_S^n.$ This Weil restriction exists because $F_S^n$ is a finite and locally free universal homeomorphism (see \cite[Chapter 7.6, Theorem 4]{BLR}; the proof works without the condition imposed there because $F_S^{n}$ is a universal homeomorphism). Recall that we have a canonical $S$-morphism $X \to X^{(p)}$ called the \it relative Frobenius \rm of $X$. Following Kato \cite{Kato}, we shall call $X\to S$ \it relatively perfect \rm if its relative Frobenius is an isomorphism. Let $\mathbf{Sch}_S$ be the category of all $S$-schemes and $\mathbf{Sch}_S^{\RP}$ the full subcategory of $\mathbf{Sch}_S$ whose objects are relatively perfect over $S.$
Then we have the following 
\begin{theorem} \rm (K. Kato) \it
The forgetful functor $\mathbf{Sch}_S^{\RP} \to \mathbf{Sch}_S$ has a right adjoint $$(-)^{\RP} \colon \mathbf{Sch}_S \to \mathbf{Sch}_S^{\RP}.$$
\end{theorem}
The \it proof, \rm as well as the construction of the right adjoint which we shall now briefly describe, can be found in \cite[\S 1]{Kato}. For each scheme $X\to S,$ we have a canonical morphism $g_{X/S} \colon X^{(1/p)} \to X$ which is the composition of the relative Frobenius $X^{(1/p)} \to (X^{(1/p)})^{(p)}$ with the adjunction morphism $(X^{(1/p)})^{(p)} \to X.$ This map turns out to be affine \cite[p. 130]{Kato}, and we have a canonical isomorphism 
$$X^{\RP} = \varprojlim ( X \leftarrow X^{(1/p)} \leftarrow X^{(1/p^2)} \leftarrow ...). $$
The canonical morphism $X^{\RP} \to X$ is an isomorphism if and only if $X$ is relatively perfect over $S;$ if $X\to S$ is locally of finite presentation, this happens if and only if $X$ is étale over $S$ \cite[Lemma 1.3]{Kato}. If $X$ is a group scheme over $S$ then so is $X^{\RP}$ and the map $X^{\RP} \to X$ is an homomorphism. If $X$ is quasi-compact and quasi-separated, and $Y$ is any scheme locally of finite presentation over $S,$ we have a canonical bijection
$$\Mor_S(X^{\RP}, Y^{\RP}) = \varinjlim \Mor_S(X^{(1/p^n)}, Y);$$ if $X$ and $Y$ are group schemes over $S,$ we similarly have $$\Hom_S(X^{\RP}, Y^{\RP}) = \varinjlim \Hom_S(X^{(1/p^n)}, Y).$$ This follows from the universal property of the relative perfection together with \cite[Tag 01ZC]{Stacks}. 

\begin{definition} \rm (Cf. \cite[Definition 8.1]{BS}, \cite[p. 5]{OS}) \it A morphism $X\to S$ is\\
(i) \rm locally relatively perfectly of finite presentation \it if $X$ admits an open covering by schemes of the form $U^{\RP},$ where $U\to S$ is of finite presentation,\\
(ii) \rm relatively perfectly of finite presentation \it if it is locally relatively perfectly of finite presentation, quasi-compact, and quasi-separated, and\\
(iii) \rm relatively perfectly smooth \it if $X$ admits an open covering by schemes of the form $U^{\RP},$ where $U\to S$ is smooth.
\end{definition}

We shall denote by $S_{\RP}$ the site whose underlying category is $\mathbf{Sch}_S^{\RP}$ endowed with the étale topology, and by $S_{\RPS}$ the full subcategory of $\mathbf{Sch}_S^{\RP}$ consisting of relatively perfectly smooth $S$-schemes, also endowed with the étale topology. The former site already appears in \cite{Kato}, and both are used in \cite{OS}. Throughout, we shall denote by $\Sh(\mathcal{S})$ (resp. $\Ab(\mathcal{S})$) the categories of sheaves of sets (resp. sheaves of Abelian groups) on a site $\mathcal{S}.$ 

We remark that the site $S_{\RP}$ is functorial in $S$ (cf. \cite[§ 5]{Kato}); the same is not generally true for $S_{\RPS}$ because for a pair of morphisms $(S'\to S, X\to S)$ with $X\to S$ relatively perfectly smooth, the same need not be true for $X\times_SS' \to S'.$ However, if the morphism $S'\to S$ is relatively perfect as well (for example, weakly étale), then we do obtain an associated morphism of sites $S'_{\RPS} \to S_{\RPS};$ this follows from \cite[Corollary 1.9]{Kato}.

Finally, let $\mathbf{Sch}_{S, \et}$ and $\mathbf{Sch}_{S, \fppf}$ be the sites consisting of all $S$-schemes, endowed with the étale and fppf-topology, respectively. We have canonical morphisms of sites $\epsilon \colon \mathbf{Sch}_{S, \fppf} \to \mathbf{Sch}_{S, \et}$ and $\iota \colon \mathbf{Sch}_{S, \et} \to S_{\RP}.$ Let $\boldsymbol{\rho}:= \iota \circ \epsilon.$ Note that, if $G$ is a group scheme over $S,$ then so is $G^{(1/p)}$ and the map $g_{G/S} \colon G^{(1/p)} \to G$ is a morphism of group schemes. Let $N:=\ker g_{G/S}.$ 

\begin{proposition} (i) We have $\boldsymbol{\rho}_\ast G = G^{\RP},$ and\\
(ii) if $G$ is smooth over $S,$ then $\R^ i \boldsymbol{\rho}_\ast G =0$ for all $i > 0$ and $\R^ i \boldsymbol{\rho}_\ast N =0$ for $i\geq 0.$ \label{RPRprop}
\end{proposition}
\begin{proof}
Claim (i) is true by definition. Since the functor $\epsilon_\ast -$ sends injective sheaves to injective sheaves, we have a Grothendieck spectral sequence
$$\R^ i \iota_\ast \R^j \epsilon_\ast G \Rightarrow \R^{i+j} \boldsymbol{\rho}_\ast G.$$
However, since $\iota_\ast -$ is exact (cf. \cite[p. 29]{BS}), this spectral sequence degenerates and yields isomorphisms 
$\iota_\ast \R^i \epsilon_\ast G = \R^i \boldsymbol{\rho}_{\ast} G $ for all $i \geq 0.$ If $G$ is smooth over $S,$ then $\R^ i \epsilon _\ast G = 0 $ for $i > 0$ by a theorem of Grothendieck (see \cite[Théorème 11.7]{Dix} and \cite[Chapter III, Remark 1.17]{Milne}), so the first part of (ii) follows. For the second part, observe that the map $\R ^i \boldsymbol{\rho}_\ast G^{(1/p)} \to \R ^i \boldsymbol{\rho}_\ast G$ is an isomorphism for all $i \ge 0,$ so the long exact cohomology sequence shows the vanishing of $\R^i \boldsymbol{\rho}_\ast N$ for $ i \geq 0.$ 
\end{proof}
\subsection{Duality}
Once again we let $S$ be a regular excellent scheme over $\F_p$ such that $\Omega^1_{S}$ is locally free of rank $r\in \N.$ Let $X$ be a scheme over $S.$ For each $n\in \N,$ we have the \it scheme of Witt vectors of length \rm $n,$ denoted by $W_n X,$ and a $\Z$-differential graded algebra of coherent $\Og_{W_nX}$-modules $W_n \Omega^\bullet_{X},$ whose differential we shall call $\rmd \, ,$ and whose product $W_n \Omega^i_X \otimes_{\Z} W_n \Omega ^j _ X \to W_n \Omega^{i+j} _X$ we shall denote by $\wedge.$ We have $W_1 \Omega^\bullet_X = \Omega^\bullet_X$ and $W_n \Omega^0_X=W_n\Og_X = \Og_{W_nX}.$ We refer the reader to \cite{Ill} for the construction of these objects. The assignment $E \mapsto W_n \Omega^q_E$ for $E$ étale over $X$ defines a sheaf on the small étale site of $X;$ patching all these sheaves together defines a sheaf on the big étale site of $X.$ We shall only be interested in those sheaves restricted either to the small étale site of $S$ or the site $S_{\RP}.$ Moreover, for a local section $x$ of $\Og_X,$ we shall denote by $\underline{x}$ the Teichmüller representative\footnote{Referred to as \it représentant multiplicatif \rm in \cite{Ill}.} of $x$ in $W_n\Og_X$ \cite[p. 505]{Ill}. For any $S$-scheme $X,$ we obtain a morphism $\mathbf{G}_{\mathrm{m}}^{\otimes q} \to W_n \Omega^q_X$ on the big étale site of $X,$ given by 
$$s_1 \otimes ... \otimes s_q \mapsto \rmd\log \underline{s_1} \wedge ... \wedge \rmd\log \underline{s_q}.$$ Note that this expression makes sense because the map $s \mapsto \rmd\log\underline{s} := \underline{s}^{-1}\rmd \underline{s}$ defines an homomorphism $\Gm \to W_n\Omega^1_X.$ We shall denote the restriction of its image to $S_{\RP}$ by $\nu_n(q)_S.$

We observe that the direct sum $\nu_n(0) \oplus ... \oplus \nu_n(r)$ naturally carries the structure of a sheaf of graded-commutative algebras on $S_{\RP}.$ Indeed, this direct sum is naturally contained in $\oplus_q W_n \Omega^q_S$ and is closed under the operation $(x,y) \mapsto x \wedge y.$ 

Following Kato \cite[Definition 4.2.3]{Kato}, we let $D(S_{\RP}, \Z/p^n\Z)$ be the derived category of sheaves of $\Z/p^n\Z$-modules on $S_{\RP}$ and $D_0(S_{\RP}, \Z/p^n\Z)$ be the smallest triangulated subcategory of $D^b(S_{\RP}, \Z/p^n\Z)$ which contains sheaves of the form $\boldsymbol{\rho}_\ast M,$ where $M$ is a finite locally free $\Og_S$-module. A central tool we shall use is the following result, due to Kato \cite[Definition 4.2.3, Theorem 4.3]{Kato}.
\begin{theorem}
The sheaves $\nu_n(q)$ are contained in $D_0(S_{\RP}, \Z/p^n\Z).$ Moreover, the functor $\R\sHom(-, \nu_n(r)_S)$ maps $D_0(S_{\RP}, \Z/p^n\Z)$ into itself. Finally, for any object $E$ of $D_0(S_{\RP}, \Z/p^n\Z),$ the canonical morphism 
$$E \to \R\sHom(\R\sHom(E, \nu_n(r)), \nu_n(r))$$ is an isomorphism. \label{Katoequivthm}
\end{theorem}

For later use, we shall now construct smooth commutative $S$-group schemes $\boldsymbol{\nu}_1(0)_S, ..., \boldsymbol{\nu}_1(r)_S$ such that $\boldsymbol{\nu}_1(0)_S^{\RP}, ..., \boldsymbol{\nu}_1(r)_S^{\RP}$ represent the sheaves $\nu_1(0)_S,...,\nu_1(r)_S$ on $S_{\RP}.$ Here we shall restrict to the case $n=1;$ it seems likely that this can be generalised to arbitrary $n,$ but this will not be necessary for the arguments which follow. For a cochain complex $A^\bullet$ of sheaves on some site $\mathcal{S},$ we shall denote by $B^q A^\bullet$ and $Z^q A^\bullet$ the sheaves of coboundaries and cocycles contained in $A^q$ for all $q\in \Z.$ Let $F:=F_S$ be the absolute Frobenius of $S.$ Recall from \cite[p. 574]{Shi} that we have the Cartier isomorphisms
$$C: H^q(F_\ast \Omega^\bullet_S) \to \Omega^q _S$$ of $\Og_S$-modules for all $q\geq 0.$ 
\begin{lemma} \label{locfreelem}
For all $q\geq 0,$ the coherent $\Og_S$-modules $Z^qF_\ast \Omega_S^\bullet$ and $B^qF_\ast \Omega^\bullet_S$ are locally free.
\end{lemma}
\begin{proof}
The two claims are analogous, so we only consider $Z^qF_\ast \Omega_S^\bullet.$ This is well-known if $S$ is smooth over $\F_p$ \cite[Proposition 0.2.2.8 (a)]{Ill}. To deduce the general case, we use an argument due to Shiho: by \cite[Remark 2.16]{Shi}, we can find, for any $s \in S,$ a flat relatively perfect morphism $\pi \colon \Spec\widehat{\Og}_{S,s} \to \mathbf{A}^m_{\F_p}$ for some $m\in \N_0;$ here we use that $S$ is contained in the \it category $\mathcal{C}$ \rm defined in \cite{Shi}; cf. \it op. cit., \rm Proposition 2.22. Moreover, we observe that the morphism $\xi \colon \Spec \widehat{\Og}_{S,s} \to S$ is flat and relatively perfect; this follows from \cite[Proposition 2.15 (2)]{Shi} together with \cite[Tag 0F6W]{Stacks}. Since both $\pi$ and $\xi$ are flat and relatively perfect, we obtain
\begin{align*}
\xi^\ast Z^q F_\ast \Omega^\bullet_S &= Z^q F_\ast \Omega^\bullet_{\Spec \widehat{\Og}_{S,s}} \\ & = \pi^\ast Z^q F_\ast \Omega^\bullet_{\mathbf{A}^m_{\F_p}},
\end{align*}
where the letter $F$ stands for the absolute Frobenius of the various relevant schemes. The functors $\pi^\ast-$ and $\xi^\ast-$ commute with $F_\ast-$ as well as with forming $\Omega^\bullet_{-}$ because $\pi$ and $\xi$ are relatively perfect. Because $s \in S$ was chosen arbitrarily, this implies the claim.
\end{proof}

We can now construct smooth commutative group schemes $\boldsymbol{\nu}_1(0)_S,..., \boldsymbol{\nu}_1(r)_S$ over $S$ which, after relative perfection, represent the sheaves $\nu_1(0)_S,..., \nu_1(r)_S.$ Recall that, for a locally free coherent sheaf $\mathscr{E}$ on $S,$ the functor on $\mathbf{Sch}_S$ given by
$$(y \colon Y\to S) \mapsto \Gamma(Y, y^\ast \mathscr{E})$$ is representable by the smooth separated $S$-group scheme $\boldsymbol{\Spec} \, \mathrm{Sym} \, \mathscr{E}^\vee,$ where, as usual, $\mathscr{E}^\vee:= \sHom_{\Og_S}(\mathscr{E}, \Og_S).$ We shall denote this group scheme by $\underline{\mathscr{E}}.$ Observe that we have a canonical isomorphism 
$$\underline{F_{S, \ast} \mathscr{E}} = \underline{\mathscr{E}}^{(1/p)}.$$ Define smooth $S$-group schemes $\boldsymbol{Z}_S^q:=\underline{Z^q F_\ast \Omega_S^\bullet},$ and $\boldsymbol{\Omega}_S^q:=\underline{\Omega_S^q}.$ We now construct a morphism
$W ^ \ast \colon  \boldsymbol{Z}_S^q \to \boldsymbol{\Omega}_S^q$ as the composition of the natural closed immersion $\boldsymbol{Z}_S^q \subseteq \underline{F_\ast \Omega^q_S}$ and the canonical morphism 
$$g_{\boldsymbol{\Omega}^q_S/S} \colon \underline{F_\ast \Omega^q_S} = (\boldsymbol{\Omega}^q_S)^{(1/p)} \to \boldsymbol{\Omega}^q_S,$$ and put
$$\boldsymbol{\nu}_1(q)_S:=\ker (C- W^\ast),$$
where $C \colon \boldsymbol{Z}^q_S \to \boldsymbol{\Omega}^q_S$ is induced by the Cartier operator. Then we have the following
\begin{proposition}
The group schemes $\boldsymbol{\nu}_1(q)_S$ are affine and smooth of relative dimension ${{r-1}\choose{q-1}} (p^r-1)$ over $S$ for all $q\geq 0.$\footnote{For $m,n\in \Z$, the binomial coefficient $m\choose n$ is defined as $\frac{m!}{(m-n)!n!}$ if $0 \leq n \leq m$ and 0 otherwise.} \label{dimprop}
\end{proposition}
\begin{proof}
The first assertion is clear. Note that we have a canonical isomorphism $\Lie \boldsymbol{\Omega}^q_S = \Omega^q_S.$ Because the morphism induced by $W^\ast$ on Lie algebras is trivial, the map on Lie algebras induced by $C-W^\ast$ is equal to $C.$ In particular, the map $\Lie \boldsymbol{Z}^q_S = Z^qF_\ast \Omega^\bullet_{S}\to \Lie \boldsymbol{\Omega}^q_S = \Omega^q_S$ is surjective by the Cartier isomorphism, which shows that $\boldsymbol{\nu}_1(q)_S$ is smooth over $S$ \cite[Proposition 1.1 (e)]{LLR} (using that $\boldsymbol{Z}^q_S$ and $\boldsymbol{\Omega}^q_S$ are smooth over $S$). Moreover, the exact sequence
$$0 \to B^qF_\ast \Omega^\bullet_S \to Z^qF_\ast\Omega^\bullet_S \overset{C}{\to} \Omega_{S}^q\to 0$$ provides a canonical isomorphism $\Lie \boldsymbol{\nu}_1(q)_S =  B^qF_\ast\Omega^\bullet_S.$ Hence all that remains to be calculated is $\rk_{\Og_S} B^qF_\ast\Omega^\bullet_S.$ Put $b_q:=\rk_{\Og_S} B^qF_\ast\Omega^\bullet_S$ and $z_q:=\rk_{\Og_S} Z^qF_\ast\Omega^\bullet_S.$ We clearly have $b_{q+1} + z_q = {r\choose{q}}p^r,$ and the Cartier isomorphism gives $z_q - b_q = {r \choose q}.$ We obtain the recursion $b_{q+1}= {r \choose q}(p^r-1)-b_q$ for $q\geq 0,$ from which the claim follows by induction using that $B^0F_\ast\Omega^\bullet_S=0.$
\end{proof}

As indicated earlier, we now have

\begin{proposition}
The group schemes $\boldsymbol{\nu}_1(0)_S^{\RP},$..., $\boldsymbol{\nu}_1(r)_S^{\RP}$ represent the sheaves $\nu_1(0)_S,$..., $\nu_1(r)_S$ on $S_{\RP}.$ \label{representprop}
\end{proposition}
\begin{proof}
The morphism $g_{G/S}^{\RP}$ is the identity on $G^{\RP}$ for any $S$-group scheme $G.$ In particular, the morphism $(C-W^\ast)^{\RP}$ is equal to the map $C-1 \colon Z^q \Omega^\bullet_S \to \Omega^q_S$ from \cite[p. 136, (3.1.4)]{Kato}. Therefore, sequence (3.1.4) from \it loc. cit. \rm implies the result. 
\end{proof}

In order to gain some intuition as to how these smooth group schemes behave, we give the following examples: 

\begin{itemize}
\item We have $\boldsymbol{\nu}_1(0)_S = \F_p.$ Indeed, for $q=0,$ the Cartier operator $C\colon \Og_S = Z^0 F_\ast \Omega_S^\bullet \to \Omega^0_S=\Og_S$ is the identity, and $W^\ast$ equals the composition $\Ga \to \mathbf{G}_{\mathbf{a}}^{(1/p)} \overset{g_{\Ga/S}}{\to} \Ga,$ which is the relative Frobenius $F_{\Ga/S}.$ Hence $C-W^\ast = \Id_{\Ga}-F_{\Ga/S},$ which implies the claim. 

\item The group scheme $\boldsymbol{\nu}_1(1)_S$ is canonically isomorphic to $\mathbf{G}_{\mathrm{m}}^{(1/p)}/\Gm.$\footnote{The quotient $\mathbf{G}_{\mathrm{m}}^{(1/p)}/\Gm$ is representable because it becomes isomorphic to the \it affine \rm $S$-scheme $\mathbf{A}^{p^r-1}_S$ after the fppf-base change $F_S\colon S\to S.$} Write $F:=F_S.$ We shall first construct a morphism $\partial_0\colon \mathbf{G}_{\mathrm{m}}^{(1/p)} \to \underline{F_\ast \Omega^1_S}.$ Let $N_0:=\ker g_{\Ga/S},$ where $g_{\Ga/S} \colon \mathbf{G}_{\mathrm{a}}^{(1/p)} \to \Ga$ is the canonical morphism. Then the natural open immersion $\mathbf{G}_{\mathrm{m}}^{(1/p)} \subseteq \mathbf{G}_{\mathrm{a}}^{(1/p)}$ identifies $\mathbf{G}_{\mathrm{m}}^{(1/p)}$ with $\mathbf{G}_{\mathrm{a}}^{(1/p)} \backslash N_0.$ \footnote{This observation is due to J. Oesterlé \cite[p. 72]{Oes} in the case $r=1,$ who uses explicit equations. Let $T$ be a scheme over $S,$ and consider a commutative diagram
$$\begin{tikzcd}[ampersand replacement=\&]
 T \arrow{r}{\phi} \arrow[swap]{d}{F_{T/S}} \& \mathbf{G}_{\mathrm{a}}^{(1/p)} \arrow[d] \arrow{rd}{g_{\Ga/S}}\\
T^{(p)} \arrow[swap]{r}{\phi^{(p)}} \& (\mathbf{G}_{\mathrm{a}}^{(1/p)})^{(p)} \arrow[swap]{r}{\alpha} \& \Ga,
\end{tikzcd}$$ where $\alpha$ is the adjunction morphism and $\phi$ corresponds to $\alpha\circ \phi^{(p)}$ under the bijection $\Hom_S(T, \mathbf{G}_{\mathrm{a}}^{(1/p)})=\Hom_S(T^{(p)}, \Ga).$ Since $F_{T/S}$ is an homeomorphism, $\alpha\circ \phi^{(p)}$ is no-where vanishing if and only if $\phi$ factors through $\mathbf{G}_{\mathrm{a}}^{(1/p)}\backslash N_0.$}
In particular, the function $g_{\Ga/S}$ is invertible on $\mathbf{G}_{\mathrm{m}}^{(1/p)}.$ Moreover, the universal derivation $\rmd \colon F_\ast \Og_S \to F_\ast \Omega^1_S$ is $\Og_S$-linear and hence defines a morphism of group schemes $\rmd \colon \mathbf{G}_{\mathrm{a}}^{(1/p)} \to \underline{F_\ast \Omega^1_S}.$ We put $$\partial_0 := g_{\Ga/S}^{-1} \rmd \colon \mathbf{G}_{\mathrm{m}}^{(1/p)} \to \underline{F_\ast \Omega^1_S}.$$ First note that we have $\partial_0(y)= \rmd\log y$ for any étale-local section $y$ of $\mathbf{G}_{\mathrm{m}}^{(1/p)}$ over $S$ by construction. In particular, $\partial_0$ is a morphism of group schemes (as étale-local sections are fibre-wise dense in smooth $S$-schemes) and the image of the canonical closed immersion $\Gm \to \mathbf{G}_{\mathrm{m}}^{(1/p)}$ is contained in $\ker \partial_0,$ thus giving rise to a morphism $\partial \colon \mathbf{G}_{\mathrm{m}}^{(1/p)}/\Gm \to \underline{F_\ast \Omega^1_S}.$ This morphism clearly factors through $\boldsymbol{Z}^1_S,$ and after restricting to the small étale site of $S,$ $\im \partial$ coincides with $\nu_1(1)_S.$ Therefore we obtain a surjective morphism 
$$\partial \colon \mathbf{G}_{\mathrm{m}}^{(1/p)}/\Gm \to \boldsymbol{\nu}_1(1)_S.$$ An easy calculation shows that the map on Lie algebras induced by $g_{\Ga/S}$ vanishes; hence the same is true for $g_{\Ga/S}^{-1}.$ Since $g_{\Ga/S}(1)=1,$ the map
$$\Lie \partial_0 \colon F_\ast \Og_S = \Lie \mathbf{G}_{\mathrm{m}}^{(1/p)} \to \Lie \boldsymbol{\nu}_1(1)_S = B^1F_\ast \Omega^\bullet_S$$ coincides with the universal derivation $\rmd \,,$ which shows in particular that $\Lie \partial$ is an isomorphism. Therefore $\partial$ is étale, which means that it is an isomorphism as soon as $\partial^{\RP}$ is. This follows from Proposition \ref{representprop}.

\item Finally, choose a $p$-basis $t_1,..., t_r$ of $S;$ this is possible after shrinking $S$ if necessary since $\Omega^1_S$ is locally free of rank $r.$ Then the group scheme $\boldsymbol{\nu}_1(r)_S$ is cut out inside $\mathbf{G}_{\mathrm{a}}^{p^r}$ by the $p$-polynomial
$$x_{0...0(p-1)} - \sum_{0 \leq i_1,..., i_r \leq p-1} t_1^{i_1}\cdot ... \cdot t_r^{i_r}x_{i_1... i_r}^p.$$ Here, the coordinates of $\mathbf{G}_{\mathrm{a}}^{p^r}$ are $(x_{i_1...i_r})_{0 \leq i_1,..., i_r \leq p-1};$ see section 8 of \cite{OS} for this calculation. 
\end{itemize}
\subsection{Duality over a field}
In this subsection, we let $S= \Spec \kappa$ for some field $\kappa.$ Our assumptions on $S$ are then equivalent to $[\kappa:\kappa^p]=p^r.$ We shall denote by $\sHom_{\kappa_{\RP}}(-,-)$ the internal Hom functor of Abelian sheaves on $(\Spec \kappa)_{\RP}.$ We begin by recalling several definitions and results from \cite{BS, Kato, Suz}. A commutative group scheme $G$ relatively perfectly of finite presentation over $\kappa$ is always of the form $G_0^{\RP}$ for some \it smooth \rm algebraic group $G_0$ over $\kappa$ by \cite[p. 28]{BS}; it is said to be \it unipotent \rm if such a $G_0$ can be chosen to be unipotent (this definition is different from the one used in \cite{Suz}, but they are equivalent by Proposition 3.1 of \it op. cit.\rm). Finally, we say that $G$ is \it wound unipotent \rm if $G_0$ can be chosen to be wound unipotent\footnote{We shall call a smooth unipotent algebraic group $G_0$ over $\kappa$ \it wound unipotent \rm if $\Hom_{\kappa}(\Ga, G_0)=0.$ In particular, $G_0$ need not be connected. This is the same terminology as in \cite{Suz} but differs from that of \cite{CGP} and \cite{Ros}.} or, equivalently, if $\Hom_{\kappa_{\RP}}(\mathbf{G}_{\mathrm{a}}^{\RP}, G)=0.$ Note that, for $n\leq m \in \N_0,$ we have a canonical injection $\nu_n(r)_\kappa \to \nu_m(r)_\kappa$ \cite[Lemma 4.1.5]{Kato}, and we put
$$\nu_{\infty}(r)_\kappa := \varinjlim \nu_n(r)_\kappa.$$ Let $D(\kappa_{\RP})$ be the derived category of sheaves of Abelian groups on $\kappa_{\RP},$ and let $D_0(\kappa_{\RP})$ be the smallest triangulated subcategory of $D^b(\kappa_{\RP})$ which contains $\Ga$ (regarded as a complex concentrated in degree 0; cf. \cite{Kato, Suz}). For a group scheme $G$ relatively perfectly of finite presentation over $\kappa,$ we have
\begin{align}\R\sHom_{\kappa_{\RP}}(G, \nu_{\infty}(r)_{\kappa}) = \varinjlim \R\sHom_{\kappa_{\RP}}(G, \nu_{n}(r)_{\kappa})\label{Rlim}\end{align} by \cite[p. 11]{Suz}. As in \cite[p. 139]{Kato}, we observe that we have an exact functor
$$\theta\colon D_0(\kappa_{\RP}, \F_p) \to D_0(\kappa_{\RP})$$ of triangulated categories, as well as a functor $\theta^!:=\R\sHom_{\kappa_{\RP}}(\F_p,-)$ in the opposite direction which satisfies
$$\theta(\R\sHom_{\F_p}(E, \theta^!(F))) = \R\sHom_{\kappa_{\RP}}(\theta(E), F)$$ for all $E\in D_0(\kappa_{\RP}, \F_p),$ $F\in D_0(\kappa_{\RP}).$ The free resolution $0 \to \Z \to \Z \to \F_p\to 0$ together with the exact sequence $0 \to \nu_1(r) \to \nu_{\infty}(r) \to \nu_{\infty}(r) \to 0$ shows that we have an isomorphism
$$\nu_1(r) \cong \theta^! (\nu_{\infty}(r)).$$ In particular, fixing such an isomorphism once and for all, we obtain 
$$\theta(\R\sHom_{\F_p}(E, \nu_1(r))) = \R\sHom_{\kappa_{\RP}}(\theta(E), \nu_{\infty}(r))$$ for all $E\in D_0(\kappa_{\RP}, \F_p),$ first for $E=\Ga$ and then, by induction, for all $E.$ 
Hence, exactly as in \cite[Theorem 4.3]{Kato}, we see that $\R\sHom_{\kappa_{\RP}}(-, \nu_{\infty}(r))$ induces a contravariant autoequivalence of $D_0(\kappa_{\RP}).$ This observation is already used in \cite{Suz}. The following result is a refinement of Theorem \ref{Katoequivthm}:
\begin{proposition} \rm (Suzuki) \it For a commutative group scheme $G$ relatively perfectly of finite presentation over $\kappa,$ we denote the element of $D(\kappa_{\RP})$ consisting of $G$ in degree 0 also by $G.$ Then the following claims hold: \\
(i) The category $D_0(\kappa_{\RP})$ consists precisely of those (bounded) complexes whose cohomology objects are representable by commutative unipotent group schemes relatively perfectly of finite presentation over $\kappa.$ \\
(ii) Suppose that $G$ is a commutative wound unipotent group scheme relatively perfectly of finite presentation over $\kappa.$ Then $\R\sHom_{\kappa_{\RP}}(G, \nu_{\infty}(r)_{\kappa})$ is concentrated in degree 0. Moreover, $$G^\vee:=\sHom_{\kappa_{\RP}}(G, \nu_{\infty}(r)_\kappa)$$ is representable by a wound unipotent group scheme relatively perfectly of finite presentation over $\kappa.$ The canonical morphism $G\to (G^\vee)^\vee$ is an isomorphism. \\
(iii) If $0 \to G_1\to G_2 \to G_3 \to 0$ is an exact sequence of commutative \rm wound unipotent \it group schemes relatively perfectly of finite presentation over $\kappa,$ then the sequence $0 \to G_3^\vee \to G_2^\vee \to G_1^\vee \to 0$ is exact. 
\end{proposition}
\begin{proof}
Part (i) is \cite[Proposition 3.2]{Suz}. The first part of (ii) follows from \cite[Propositions 3.3 and 3.4]{Suz}, the second and third from \cite[Proposition 3.5]{Suz}. Finally, (iii) follows from the vanishing of $\sExt^1_{\kappa_{\RP}}(G_1, \nu_{\infty}(r)_{\kappa})$ contained in part (ii). 
\end{proof}\\
\\
$\mathbf{Remark}.$ There is a similar duality for commutative split unipotent groups, which will not be used in this article. See \cite[Section 3]{Suz} for more details. \\

For $0 \leq q \leq r,$ we have $(\boldsymbol{\nu}_1(q)_{\kappa}^{\RP})^\vee = \boldsymbol{\nu}_1(r-q)_{\kappa}^{\RP};$ see \cite[Theorem 4.3]{Kato} or \cite[p. 12]{Suz}. This shows that every $\boldsymbol{\nu}_1(q)_{\kappa}^{\RP}$ (and hence every $\boldsymbol{\nu}_1(q)_{\kappa}$) is wound unipotent. Indeed, an homomorphism $\mathbf{G}_{\mathrm{a}}^{\RP} \to \boldsymbol{\nu}_1(q)_{\kappa}^{\RP}$ is the same as a $\Z$-bilinear map $\mathbf{G}_{\mathrm{a}}^{\RP} \times_\kappa \boldsymbol{\nu}_1(r-q)_{\kappa}^{\RP}\to \boldsymbol{\nu}_1(r)_{\kappa}^{\RP}.$ Since the $\kappa\sep$-points are dense in $\boldsymbol{\nu}_1(r-q)_{\kappa}^{\RP}$ by \cite[Proposition 8.5]{BS} and $\boldsymbol{\nu}_1(r)_{\kappa}^{\RP}$ is wound unipotent by \cite[Theorem 3.2 (ii)]{Kato}, any such map vanishes. 

\section{Weakly permawound groups}
As before, let $\kappa$ be a field of characteristic $p>0$ such that $[\kappa:\kappa^p]=p^r$ for some $r\in \N_0.$ Let $G$ be a commutative smooth unipotent algebraic group over $\kappa.$ We shall first recall the definition of \it weakly permawound \rm algebraic groups, recently introduced by Z. Rosengarten \cite[Definition 5.1]{Ros}. These are characterised by the fact that they admit filtrations whose subquotients are cut out inside some $\mathbf{G}_{\mathrm{a}}^{d}$ by a $p$-polynomial whose principal part is \it universal \rm \cite[Definition 3.1, Proposition 6.4]{Ros}. We shall then classify weakly permawound algebraic groups \it up to relative perfection; \rm this will generalise Theorem 10.5 from \cite{OS}. 

\begin{definition} \rm (Rosengarten) \it (i) Let $F = F_1(X_1) + ... + F_n(X_n)$ be a $p$-polynomial in $\kappa[X_1,..., X_n]$ for some $n\in \N.$ Recall that the \rm principal part \it $P$ of $F$ is the sum of the leading monomials of the $F_j.$ We say that $P$ is \rm universal \it if the homomorphism $\kappa^n \to \kappa$ given by $P$ is surjective. \\
(ii) Let $G$ be a commutative unipotent algebraic group over $\kappa.$ Then $G$ is \rm weakly permawound \it if for all fppf-exact sequences 
$$ G \to E \to \Ga \to 0$$ of commutative $\kappa$-group schemes, we have $\Hom_{\kappa}(\Ga, E) \not=0.$ 
\end{definition}

\noindent$\mathbf{Example}.$ One sees immediately that $\Ga$ is weakly permawound \cite[Proposition 5.3]{Ros}. The principal example of a smooth \it wound unipotent \rm weakly permawound algebraic group is $\boldsymbol{\nu}_1(r)_{\kappa}.$ Indeed,  $\boldsymbol{\nu}_1(r)_{\kappa}$ is isomorphic to the smooth algebraic group $\mathscr{V}$ defined in \cite[Definition 7.3]{Ros}, which is weakly permawound by \cite[Theorem 6.10]{Ros}. This follows from the presentation of $\boldsymbol{\nu}_1(r)_{\kappa}$ given in the third example following Proposition \ref{representprop}.

\begin{lemma}
Let $G$ be a commutative smooth wound unipotent algebraic group over $\kappa.$ Then $G$ is weakly permawound if and only if so is $G^{(1/p)}.$ 
\end{lemma}
\begin{proof}
By \cite[Proposition 6.7]{Ros}, we may assume that $\kappa$ is separably closed. We shall use the following observation used (implicitly) in \cite{Ros}: the group $G$ is weakly permawound if and only if $\Ext^1_\kappa(\Ga, G)$ is finite-dimensional over $\kappa.$ Indeed, if $G$ is weakly permawound, it admits a finite filtration with successive quotients $N_0$ or $\boldsymbol{\nu}_1(r)_\kappa$  by \cite[Theorem 9.5]{Ros}. Since both those groups are cut out by reduced $p$-polynomials with universal principal part, $\Ext^1_\kappa(\Ga, G)$ is finite-dimensional by \cite[Proposition 4.2]{Ros} and dévissage. Conversely, there exists a finite filtration of $G$ whose successive quotients are smooth, wound, and annihilated by $p$ (this follows from the connected case \cite[Chapter 10.2, Lemma 12]{BLR}). Each successive quotient is cut out inside some power of $\Ga$ by a reduced $p$-polynomial by \cite[Lemma B.1.7 and Proposition B.1.13]{CGP}, whose principal parts must be universal by \cite[Lemma 4.7]{Ros}. Hence $G$ is weakly permawound by \cite[Lemma 6.4]{Ros}. Now observe that, since $F_\kappa$ is finite, its higher derived images in the étale topology vanish. Therefore, we have a canonical isomorphism
$$\Ext^i_\kappa (\Ga, G^{(1/p)})=F_{\kappa \ast} \Ext_\kappa^i (\Ga, G)$$ of coherent sheaves on $\Spec \kappa$ for all $i\geq 0.$ This implies the claim.  
\end{proof}

We are now ready to prove 
\begin{proposition}
Suppose that $\kappa$ is separably closed. Let $G$ be a commutative smooth wound unipotent algebraic group over $\kappa.$ Then the following are equivalent: \\
(i) $G$ is weakly permawound,\\
(ii) $G^{\RP}$ admits a finite filtration $0=G_0 \subseteq ... \subseteq G_n=G^{\RP}$ such that $G_{j+1}/G_j \cong \boldsymbol{\nu}_1(r)_\kappa^{\RP}$ for $j=0,..., n-1,$ and \\
(iii) the dual $(G^{\RP})^\vee$ of $G^{\RP}$ is étale over $\kappa.$ \label{étalepwprop}
\end{proposition}
\begin{proof}
(i)$\Rightarrow$(ii): By \cite[Theorem 9.5]{Ros}, $G$ admits a finite filtration $0=U_0 \subseteq U_1 \subseteq ... U_d = G$ such that $U_{j+1}/U_j \cong \boldsymbol{\nu}_1(r)_\kappa$ or $U_{j+1}/U_j\cong N_0$ for all $j=0,..., d-1.$ The left exactness of relative perfection yields a finite filtration $0=U_0^{\RP} \subseteq ... \subseteq U_d^{\RP} =G^{\RP}.$ We shall proceed by induction on $d,$ the case $d=0$ being trivial. If $U_1\cong N_0,$ then the map $G^{\RP} \to (G/U_1)^{\RP}$ is an isomorphism by Proposition \ref{RPRprop} (ii). If $U_1\cong \boldsymbol{\nu}_1(r)_\kappa,$ then the sequence
$$0 \to U_1^{\RP} \to G^{\RP} \to (G/U_1)^{\RP} \to 0$$ is exact. In both cases, we reduce to the case of a filtration of length $d-1.$ \\
(ii)$\Rightarrow$(iii): Because the duality on commutative wound unipotent groups relatively perfectly of finite presentation over $\kappa$ is exact, we obtain a filtration on $G^\vee$ whose successive quotients are isomorphic to the dual of $\boldsymbol{\nu}_1(r)_\kappa^{\RP}.$ However, this dual is isomorphic to $\F_p,$ so the claim follows. \\
(iii)$\Rightarrow$(i): We can find a finite filtration $0=U_0 \subseteq U_1 \subseteq ... \subseteq U_d = G$ by smooth wound subgroups such that the graded pieces $U_{j-1}/U_j$ are smooth, wound unipotent, and annihilated by $p$ (this follows from the connected case \cite[Chapter 10.2, Lemma 12]{BLR}). Upon relative perfection, we obtain a finite filtration of $G^{\RP}$ whose graded pieces $U_{j+1}^{\RP}/U_j^{\RP}$ are annihilated by $p.$ By assumption, $(G^{\RP})^\vee$ is étale, therefore so are the $(U_{j+1}^{\RP}/U_j^{\RP})^\vee.$ In particular, we must have 
$$U_{j+1}^{\RP}/U_j^{\RP} \cong (\boldsymbol{\nu}_1(r)_\kappa^{\RP})^{n_j}$$ for some $n_j\in \N_0;$ here we use that $\kappa$ is separably closed. We obtain surjective morphisms
$$(\boldsymbol{\nu}_1(r)_\kappa^{n_j})^{(1/p^{m_j})} \to U_{j+1}/U_j$$ for some $m_j\gg 0.$ This shows that the $U_{j+1}/U_j$ are weakly permawound \cite[Proposition 5.4]{Ros}; hence so is $G$ by \cite[Proposition 5.6]{Ros}.
\end{proof}
\begin{corollary}
Let $G$ be a commutative smooth wound unipotent algebraic group over $\kappa.$ Then $G$ is weakly permawound if and only if $(G^{\RP})^\vee$ is étale over $\kappa.$ In particular, the category of relative perfections of commutative wound unipotent weakly permawound algebraic groups over $\kappa$ is canonically equivalent to the category of commutative $p$-primary finite étale group schemes over $\kappa.$ 
\end{corollary}
\begin{proof}
We have a natural isomorphism $$(G^{\RP})^\vee \times_\kappa \Spec \kappa\sep = ((G\times_\kappa\Spec \kappa\sep)^{\RP})^{\vee}$$ because the morphism $\Spec \kappa\sep \to \Spec \kappa$ is weakly étale. Whether a commutative smooth wound unipotent algebraic group $G$ is weakly permawound can be checked after base change to $\kappa\sep$ \cite[Proposition 6.7]{Ros}; the same is evidently true for the étaleness of $(G^{\RP})^\vee.$ Hence the first part follows from Proposition \ref{étalepwprop}. The second part follows from the first and Pontryagin duality.
\end{proof} \\
\\
$\mathbf{Remark.}$ If $[\kappa:\kappa^p]=p,$ then a commutative smooth connected wound unipotent algebraic $\kappa$-group $G$ is unirational if and only if it is weakly permawound \cite[Propositions 9.6 and 9.7]{Ros2}. In particular, we recover \cite[Theorem 10.5]{OS}; the proof carries over \it verbatim. \rm
\section{Non-existence of Néron lft-models}
From now on, we shall let $S$ be an excellent Dedekind scheme over $\F_p$ for some prime number $p$ such that $\Omega^1_S$ is locally free of some rank $r\in \N.$ Let $K$ be the field of rational functions on $S.$ The goal of this section is to show that the wound unipotent algebraic groups $\boldsymbol{\nu}_1(2)_K,..., \boldsymbol{\nu}_1(r)_K$ do not admit Néron lft-models, and are therefore counterexamples to Conjecture I. The case of $\boldsymbol{\nu}_1(r)_K$ is already known \cite[Theorem 8.1]{OS}; this was deduced in \it loc. cit. \rm by means of explicit equations and a (somewhat intricate) calculation involving discrete valuations. The proof we shall give here instead relies on Shiho's results \cite{Shi} on cohomological purity and is new even in this case. We shall begin by establishing that several algebraic groups needed in this article are unirational. Note that, for each $n\in \N,$ the $K$-group scheme $(\mathbf{G}_{\mathrm{m}}^{(1/p^n)}/\Gm)^{\RP}$ represents the sheaf $\nu_n(1)_K$ on $K_{\RP}$ \cite[p. 146]{Kato}. 
\begin{proposition}
Let $G$ be a smooth algebraic group over a field $\kappa$ satisfying $[\kappa:\kappa^p]<\infty$ such that $G^{\RP}$ represents one of the sheaves $\nu_n(1)_\kappa,$..., $\nu_n(r)_\kappa$ on $\kappa_{\RP}$ for some $n\in \N.$ Then $G$ is unirational (and in particular connected). \label{unirprop}
\end{proposition}
\begin{proof}
By \cite[Theorem 7.3]{Ros2}, we may assume that $\kappa$ is separably closed. Let $1\leq q \leq r.$ Consider the $\Z$-multilinear map $\xi_0\colon \nu_n(1)_\kappa\oplus ... \oplus \nu_n(1)_\kappa \to \nu_n(q)_\kappa$ of sheaves on $\kappa_{\RP},$ where the direct sum on the left hand side consists of $q$ summands and the map is induced by the product on the sheaf of graded algebras $\nu_n(0)_\kappa \oplus ... \oplus \nu_n(r)_\kappa.$ By construction of $\nu_n(q)_\kappa,$ the image of $\xi_0(\kappa)$ generates $\nu_n(q)_\kappa(\kappa).$ In particular, we obtain a $\Z$-multilinear map 
\begin{align}\big(\mathbf{G}_{\mathrm{m}}^{(1/p^n)}/\Gm\big)^{\RP} \times_\kappa ... \times_\kappa \big(\mathbf{G}_{\mathrm{m}}^{(1/p^n)}/\Gm\big)^{\RP} \to G^{\RP},\label{map1}\end{align} which comes from a morphism of schemes
\begin{align}\xi \colon \mathbf{G}_{\mathrm{m}}^{(1/p^{n+j})}/\mathbf{G}_{\mathrm{m}}^{(1/p^{j})} \times_\kappa ... \times_\kappa \mathbf{G}_{\mathrm{m}}^{(1/p^{n+j})}/\mathbf{G}_{\mathrm{m}}^{(1/p^j)} \to G\label{map2}\end{align}
for some $j\gg 0.$ Note that the morphisms (\ref{map1}) and (\ref{map2}) restrict to the same map of sheaves on the small étale site of $\Spec \kappa.$ In particular, the image of $\xi(\kappa)$ generates $G(\kappa).$ Since $\mathbf{G}_{\mathrm{m}}^{(1/p^{n+j})}$ is an open subset of an affine space over $\kappa,$ we can now construct a morphism of schemes $\xi' \colon U \to G$ over $\kappa$ such that $U$ is an open subset of some affine space and such that the image of $\xi'(\kappa)$ generates $G(\kappa).$ Here we use that $H^1(\kappa, \mathbf{G}_{\mathrm{m}}^{(1/p^j)})=0$ by Hilbert's Theorem 90. This shows that $G$ is generated by the image of finitely many rational maps $\Ps^1_\kappa \dashrightarrow G$ through the origin. By \cite[Chapter 10.3, Theorem 2]{BLR}, each such map factors as $\Ps^1_\kappa \dashrightarrow R \to G,$ where $R$ is a unirational group over $\kappa$ and the map $R\to G$ is an homomorphism. The induced homomorphism from the product of all such $R$ is necessarily surjective, so $G$ is unirational. 
\end{proof}

We shall now show that the wound unipotent algebraic groups $\boldsymbol{\nu}_1(2)_K,..., \boldsymbol{\nu}_1(r)_K$ do not admit Néron lft-models. For a sheaf $\mathscr{F}$ on the small étale site of $S$ and a closed subscheme $D\subseteq S,$ we shall denote by $\mathscr{H}^i_D(S, \mathscr{F})$ the local cohomology sheaves of $\mathscr{F}$ with support in $D;$ cf. \cite[Tag 09XP]{Stacks}. 

\begin{lemma}
Let $D\subset S$ be a reduced closed subscheme of codimension 1. Let $U:=S\backslash D$ and let $\iota_D\colon D\to S$ and $j_U\colon U\to S$ be the canonical immersions. Then we have an exact sequence \label{Xexactlem}
$$0 \to \nu_1(q)_S \to j_{U \ast}\nu_1(q)_U \to \iota_{D\ast}\nu_1(q-1)_D\to 0$$ on the small étale site of $S$ for all $q\geq 1.$ 
\end{lemma}
\begin{proof}
We have an exact sequence
$$0  \to \nu_1(q)_S \to j_{U \ast}\nu_1(q)_U \to \iota_{D\ast} \mathscr{H}^1_D(S, \nu_1(q)_S)\to 0$$
by \cite[Tag 0A45]{Stacks} (the sheaf $\mathscr{H}^0_D(S, \nu_1(q)_S)$ of sections supported on $D$ vanishes because $\nu_1(q)_S$ is representable by an affine (hence separated) group scheme on the small étale site of $S$ by Propositions \ref{dimprop} and \ref{representprop}). Now note that $S$ is regular and that $D$ is regular of codimension 1 in $S.$ In particular, the map
$$\rho^{q, \log}_{\iota_D, 1} \colon \nu_1(q-1)_D \to \mathscr{H}^1_D(S, \nu_1(q)_S)$$ constructed in \cite[p. 590]{Shi} is an isomorphism by the cohomological purity of logarithmic Hodge-Witt sheaves \cite[p. 591, Step 1 in the proof of Theorem 3.2]{Shi}.
\end{proof}
\begin{corollary}
Let $j\colon \Spec K \to S$ be the canonical map and let $| S |$ be the set of closed points of $S.$ Then we have an exact sequence
$$0 \to \nu_1(q)_S \to j_\ast \nu_1(q)_K \to \bigoplus_{x\in | S |} \iota_{x\ast} \nu_1(q-1)_{\kappa(x)} \to 0$$
on the small étale site of $S$ for all $q\geq 1.$ \label{smexactseq}
\end{corollary}
\begin{proof}
Let $D\subseteq |S|$ be a finite subset, which we endow with the reduced subscheme structure. For an open subset $U\subset S$ and $x\in |S|,$ let $j_U\colon U\to S$ and $\iota_x \colon \Spec \kappa(x)\to S$ be the canonical immersions. By Lemma \ref{Xexactlem}, we have an exact sequence
$$0 \to \nu_1(q)_S \to j_{(S\backslash D)\ast} \nu_1(q)_{S\backslash D} \to \bigoplus_{x\in D} \iota_{x\ast} \nu_1(q-1)_{\kappa(x)} \to 0$$ on the small étale site of $S.$ Taking the inductive limit over all such $D$ (ordered by inclusion) gives the Corollary.
\end{proof}

\begin{theorem}
Assume $r>1$ and that $S$ is global (i. e. that $\#|S|=\infty$). Then, for any $2\leq q \leq r,$ the smooth wound unipotent $K$-group $\boldsymbol{\nu}_1(q)_K$ does not admit a Néron lft-model over $S.$ \label{Nétnonexthm}
\end{theorem}
\begin{proof}
Assume for the sake of a contradiction that $\mathscr{N}\to S$ is a Néron lft-model of $\boldsymbol{\nu}_1(q)_K.$ Then we obtain an exact sequence
$$0 \to \boldsymbol{\nu}_1(q)_S \to \mathscr{N} \to \bigoplus_{x\in |S|} \iota_{x\ast} \boldsymbol{\nu}_1(q-1)_{\kappa(x)} \to 0$$ on the small étale site of $S.$ Since the induced morphism $\Lie \boldsymbol{\nu}_1(q)_S \to \Lie \mathscr{N}$ is an isomorphism over $K,$ we may assume that it is an isomorphism over all of $S$ after replacing $S$ by a dense open subset if necessary (note that the condition $\#|S|=\infty$ is unaffected). This means that the morphism $\boldsymbol{\nu}_1(q)_S \to \mathscr{N}$ is étale, so its image is open in $\mathscr{N}.$ Therefore, the cokernel $\mathscr{Q}$ of the map $\boldsymbol{\nu}_1(q)_S \to \mathscr{N}$ (taken on the big fppf-site of $S$) is representable by an algebraic space étale over $S$ by \cite[Chapter 8.3, Proposition 9]{BLR}. In particular, $\mathscr{Q}$ commutes with filtered inverse limits of affine schemes \cite[Tag 04AK]{Stacks}. Because the group of connected components of the Néron lft-model of a smooth commutative algebraic group is finitely generated \cite[Proposition 3.5]{HN}, it follows that $\mathscr{Q}_{\overline{x}} = \boldsymbol{\nu}_1(q-1)_{\kappa(x)}(\kappa(x)\sep)$ is finitely generated for all $x\in |S|$ and all geometric points $\overline{x}\colon\Spec \kappa(x)\sep \to S$ mapping to $x.$ However, this implies that $\boldsymbol{\nu}_1(q-1)_{\kappa(x)}$ is étale over $\kappa(x).$ Using that $[\kappa(x):\kappa(x)^p]=p^{r-1}$ \cite[Lemma 2.1]{OvChai}, this contradicts Proposition \ref{dimprop}.
\end{proof}
\begin{corollary}
Suppose that $r>1$ and that $S$ is global. Let $G \not=0$ be a smooth commutative wound unipotent weakly permawound algebraic group over $K.$ Then $G$ does not admit a Néron lft-model over $S.$ 
\end{corollary}
\begin{proof}
For a finite separable extension $L$ of $K,$ we denote by $S_L$ the integral closure of $S$ in $L.$ Shrinking $S$ if necessary, we may assume that the morphism $S_L \to S$ is étale. Since Néron lft-models commutes with étale base change, we may therefore assume that there exists a closed immersion $\boldsymbol{\nu}_1(r)_K^{\RP} \to G^{\RP}$ by Proposition \ref{étalepwprop}. In particular, if $G$ did admit a Néron lft-model over $S,$ the same would be true for $\boldsymbol{\nu}_1(r)_K$ by \cite[Propositions 4.2, 4.5, and 4.6]{OS}, and we already know that this is not the case from Theorem \ref{Nétnonexthm}.
\end{proof}

\noindent$\mathbf{Remark}.$ The significance of the preceding result lies in the fact that, over any imperfect field $\kappa$ such that $[\kappa:\kappa^p]=p^r<\infty,$ wound unipotent weakly permawound algebraic groups are ubiquitous; see \cite[Theorem 1.4]{Ros}. In particular, as soon as $r>1,$ counterexamples to Conjecture I are very frequent.  It should also be remarked that the groups $\boldsymbol{\nu}_1(q)_K$ are not weakly permawound for $0 \leq q < r$ by Proposition \ref{étalepwprop} (their duals after relative perfection are $\boldsymbol{\nu}_1(r-q)_K^{\RP},$ which are not étale over $K$ by Proposition \ref{dimprop}). In particular, as soon as $r>2,$ there exist infinite families of counterexamples to Conjecture I which are not weakly permawound. Finally, note that all counterexamples we have constructed in this article are unirational by Proposition \ref{unirprop}.

\section{An elementary proof}
Let $S$ be a Dedekind scheme over $\F_p$ for some prime number $p.$ Let $K$ be the field of rational functions on $S.$ We denote by $r$ the unique positive integer such that $[K:K^p]=p^r$ if $[K:K^p]< \infty$ and put $r=\infty$ otherwise. The condition $r<\infty$ will be used synonymously with the requirement $[K:K^p]<\infty.$ 
We shall first recall the following known definitions and facts to be used later: 
\begin{itemize}
\item If $r<\infty$ and $S$ is \it excellent, \rm then $\Omega^1_S$ is locally free of rank $r$ (\cite[Lemma 2.1]{OvChai} together with \cite[Propositions 2.20 and 2.22, Remark 2.21]{Shi}). In particular, $F_S$ is finite and locally free of rank $p^r.$
\item Moreover, if $S$ is excellent, then the residue fields of $S$ (at closed points) are perfect if and only if $r=1$ \cite[Corollary 2.2]{OvChai}.
\item In particular, if $S$ is excellent and $r<\infty,$ the algebraic group $\boldsymbol{\nu}_1(1)_K = \mathbf{G}_{\mathrm{m}}^{(1/p)}/\Gm$ admits a Néron lft-model over $S.$ If the Néron lft-model of $\Gm$ over $S$ is denoted by $\mathscr{G}_{\mathrm{m}},$ the Néron lft-model of $\boldsymbol{\nu}_1(1)_K$ is, in fact, isomorphic to $\mathscr{G}_{\mathrm{m}}^{(1/p)} / \mathscr{G}_{\mathrm{m}}.$ See \cite[Lemma 2.9]{OvGR} for both those claims.
\end{itemize}

From now on, we shall assume that $S$ is \it excellent \rm unless explicitly stated otherwise. The goal of this second part of the present article is to give a new and more elementary proof of the following
\begin{theorem}  \rm (Cf. \cite{BLR}, Conjecture I) \it 
Assume that the residue fields of $S$ are perfect (or, equivalently, that $r=1$). Let $G$ be a smooth commutative algebraic group over $K$ such that $\Hom_K(\Ga, G)=0.$ Then $G$ admits a Néron lft-model over $S.$ \label{ExistenceThm}
\end{theorem}
This result was recently established by T. Suzuki and the author \cite[Theorem 6.3]{OS}, but the methods used in \it op. cit. \rm rely on the rather advanced theory of relatively perfect schemes, duality on relatively perfect sites, as well as the newly introduced concept of \it relatively perfect Néron models \rm \cite[Definition 4.1]{OS}. The aim of the present section is to circumvent the use of these methods; our new proof is partly inspired by ideas already present in \it op. cit., \rm but several new ideas are nevertheless needed. Most importantly, we shall make use of Rosengarten's embedding theorem \cite[Theorem 1.4]{Ros}, which allows us to embed any commutative $p$-torsion smooth connected wound unipotent group $G$ into a smooth wound weakly permawound one. This will allow us to reduce the (crucial) unipotent case to that of $\boldsymbol{\nu}_1(1)_K \cong \mathbf{G}_{\mathrm{m}}^{(1/p)}/\Gm.$ Our reliance on Rosengarten's results does not destroy the elementary nature of the present approach since the results in \cite{Ros}, while highly non-trivial, are also derived by elementary methods.

We begin by recalling the following result:
\begin{proposition}
Let $E$ be a commutative finite étale group scheme over $K.$ Then $E$ admits a Néron model over $S.$ In particular, a smooth commutative algebraic group $G$ over $K$ admits a Néron (lft-)model over $S$ if and only if so does its identity component $G^0.$ \label{etaleprop}
\end{proposition}
\begin{proof}
If $E$ is the constant $K$-group scheme associated with some finite Abelian group $H,$ one immediately verifies that the constant $S$-group scheme associated with $H$ is the Néron model of $E.$ One reduces to this case using \cite[Chapter 7.2, Proposition 4]{BLR}. 

For the second claim, let $\pi_0(G)$ be the scheme of connected components of $G.$ After replacing $K$ by a finite separable extension if necessary, we may assume that $\pi_0(G)$ is constant. Then $G=\sqcup_{\alpha\in \pi_0(G)} G_{\alpha},$ where $G_{\alpha}$ is the connected component of $G$ mapped to $\alpha.$ Replacing $K$ by a finite separable extension $L$ again, we may assume that all $G_{\alpha}$ are isomorphic to $G^0.$ The group structure of $G$ is then given by a family of morphisms $\phi_{\alpha\beta} \colon G_{\alpha}\times_K G_{\beta} \to G_{\alpha+\beta}.$ Let $S_L$ be the integral closure of $S$ in $L.$ Suppose first that $G^0$ admits a Néron lft-model $\mathscr{G}_0.$ Note that this condition is not affected by replacing $S$ by $S_L$ because the morphism $S_L\to S$ is generically étale \cite[Chapter 10.1, Proposition 9]{BLR}. Our assumption implies that each connected component $G_{\alpha}$ admits a Néron lft-model $\mathscr{G}_{\alpha}$ over $S.$ The morphisms $\phi_{\alpha\beta}$ canonically extend to morphisms $\mathscr{G}_{\alpha}\times_S \mathscr{G}_{\beta} \to \mathscr{G}_{\alpha + \beta}$ by the Néron mapping property. These morphisms define a group structure on the $S$-scheme $$\mathscr{G}:=\bigsqcup_{\alpha\in \pi_0(G)} \mathscr{G}_{\alpha}.$$ If $\Pi$ is the Néron model of the (constant) group scheme $\pi_0(G),$ the group scheme $\mathscr{G}$ fits into an exact sequence $0\to \mathscr{G}_0 \to \mathscr{G} \to \Pi \to 0$ by construction. Hence $G$ admits a Néron (lft-)model by \cite[Corollary 2.6]{OvGR} together with \cite[Chapter 7.2, Proposition 4 and Chapter 10.1, Proposition 4]{BLR}. For the other direction, let $\mathscr{G}$ be a Néron (lft-)model of $G$ over $S,$ and let $\mathscr{G}'$ be the scheme-theoretic closure of $G^0$ in $\mathscr{G}.$ Then $\mathscr{G}'$ is smooth over $S$ and hence a Néron (lft-)model of $G^0.$
\end{proof}

The following proposition is known and was used, for example, in \cite{OS}. As there does not seem to be a complete proof in the literature, we provide one here.
\begin{proposition}
Let $0 \to G' \to G \to G'' \to 0$ be an exact sequence of smooth commutative algebraic groups over $K.$ If $G'$ and $G''$ admit Néron lft-models over $S,$ then so does $G.$\label{Dévprop}
\end{proposition}
\begin{proof}
Note that the map $G^0 \to G''^0$ is smooth and surjective, but its kernel need not be connected. Using Proposition \ref{etaleprop}, we may assume that $G$ and $G''$ are connected, and that $G'$ is either connected or étale. Denote by $\mathscr{G}'$ and $\mathscr{G}''$ the Néron lft-models of $G'$ and $G'',$ respectively. Let $\mathscr{G}''^0$ be the identity component of $\mathscr{G}''.$ Moreover, let $\mathscr{N}:=\mathscr{G}'^0$ if $G'$ is connected, and $\mathscr{N}:=\mathscr{G}'$ if $G'$ is étale. Then $\mathscr{N}$ and $\mathscr{G}''^0$ are of finite presentation over $S,$ so after shrinking $S$ if necessary, we may suppose that the exact sequence from the proposition extends to an exact sequence
$$0 \to \mathscr{N} \to \mathscr{G} \to \mathscr{G}'^0 \to 0,$$ where $\mathscr{G}$ is a (necessarily smooth) model of $G$ of finite presentation over $S.$ For each closed point $s\in S,$ $G$ admits a Néron lft-model $\mathscr{G}_s$ over $\Og_{S,s}$ by \cite[Chapter 10.2, Theorem 2]{BLR}. Now pick such a closed point $s$ and denote by $\mathscr{G}'_s$ and $\mathscr{G}''_s$ the Néron lft-models of $G'$ and $G''$ over $\Og_{S,s},$ which are localisations of $\mathscr{G}'$ and $\mathscr{G}'',$ respectively. Consider the commutative diagram 
$$\begin{tikzcd}
0 \arrow[r]  & \Lie \mathscr{N} \otimes_{\Og_S} \Og_{S,s} \arrow[r]\arrow[swap]{d}{\cong} & \Lie \mathscr{G} \otimes_{\Og_S} \Og_{S,s} \arrow[r]\arrow{d} &  \Lie \mathscr{G}''^0 \otimes_{\Og_S} \Og_{S,s} \arrow[r] \arrow{d}{\cong}&0 \\
&\Lie \mathscr{G}'_s \arrow[r] & \Lie \mathscr{G}_s \arrow[r] & \Lie \mathscr{G}''_s; 
\end{tikzcd}$$
here the top row is exact, the bottom row is a complex, and the right and left vertical arrows are isomorphisms. This shows that the morphism $\Lie \mathscr{G}_s \to  \Lie \mathscr{G}''_s$ is surjective, so the map $\mathscr{G}_s \to \mathscr{G}''_s$ is smooth (this can be reduced to the finite type case by considering identity components, which is \cite[Proposition 1.1 (e)]{LLR}). In particular, the kernel of this map is the Néron lft-model of its generic fibre, so the sequence $0 \to \mathscr{G}'_s \to \mathscr{G}_s \to \mathscr{G}''_s$ is exact. This shows that the bottom row in the diagram is, in fact, exact and the bottom left horizontal arrow is injective. Therefore the lemma of five homomorphisms shows that the map 
$$ \Lie \mathscr{G} \otimes_{\Og_S} \Og_{S,s} \to \Lie \mathscr{G}_s$$ is an isomorphism for all closed points $s\in S.$ The proposition now follows from \cite[Chapter 10.1, Proposition 9]{BLR}. 
\end{proof} \\
\noindent$\mathbf{Remark}.$ It is not generally true that the sequence $0 \to \mathscr{G}' \to \mathscr{G} \to \mathscr{G}'' \to 0$ of Néron lft-models induced by the exact sequence from the proposition is exact over some dense open subscheme of $S.$ For example, let $n$ be a positive integer invertible on $S.$ Then the exact sequence $0 \to \boldsymbol{\mu}_n \to \Gm \to \Gm \to 0$ over $K$ induces the exact sequence $0 \to \boldsymbol{\mu}_n \to \mathscr{G}_{\mathrm{m}}\to \mathscr{G}_{\mathrm{m}}$ of Néron lft-models over $S;$ the last map is not surjective over any dense open subset of $S$ because the induced map on connected components is multiplication by $n$ on $\Z$ for all closed points $s\in S.$ 

The statement analogous to Proposition \ref{Dévprop} for (quasi-compact) Néron models is well-known \cite[Chapter 7.5, Proposition 1]{BLR}. We shall use both results freely and say that the Néron (lft-)model of $G$ exists \it by dévissage. \rm

\begin{proposition}
Let $\mathscr{G} \to S$ be a smooth separated commutative group scheme with generic fibre $G$ of finite presentation over $K=\Og_{S, \eta}.$ Assume that
\begin{itemize}
\item for every closed point $s\in S,$ $G$ admits a Néron lft-model $\mathscr{N}_s$ over $\Og_{S,s},$
\item for every \rm étale \it morphism $E\to S$ and every map $\phi\colon E\times_S\Spec K \to G,$ there is a unique map $E\to \mathscr{G}$ extending $\phi,$ and 
\item for every closed point $s\in S,$ the Abelian group $\mathscr{G}(\kappa(s)\sep)/\mathscr{G}^0(\kappa(s)\sep)$ is finitely generated. 
\end{itemize}
Then $\mathscr{G}$ is the Néron lft-model of $G$ over $S.$ \label{IsNmprop}
\end{proposition}
\begin{proof} The proposition follows from \cite[Lemma 2.5]{OvGR} once we can show that the morphisms $\mathscr{G}\times_S \Spec \Og_{S,s} \to \mathscr{N}_s$ are isomorphisms for all closed points $s$ of $S.$ 

Assume first that $G$ is connected. The conditions imply that, for each closed point $s\in S,$ the canonical morphism $\mathscr{G} \times_S \Spec \Og_{S,s} \to \mathscr{N}_s$ is an isomorphism; this can be deduced as in the proof of \cite[Proposition 3.11]{OvChai} (since we are not assuming that $\kappa(s)$ be perfect for this proposition, it is not in general true that $\mathrm{coker} (\nu(\kappa(s)\sep))=(\mathrm{coker}\,\nu)(\kappa(s)\sep)$ in the notation of \it loc. cit. \rm However, there is a surjective morphism $$\mathrm{coker} (\nu(\kappa(s)\sep)) \to (\mathrm{coker}\,\nu)(\kappa(s)\sep),$$ so the argument still applies). 

In general, let $\mathscr{G}'$ be the scheme-theoretic closure of $G^0$ in $\mathscr{G}.$ The conditions listed in the proposition are clearly satisfied for $\mathscr{G}'$ as well, so $\mathscr{G}'$ is the Néron lft-model of $G^0$ over $S.$ Because, for each closed point $s\in S,$ the Néron lft-model $\mathscr{N}_s'$ of $G^0$ over $\Og_{S,s}$ is an open subscheme of $\mathscr{N}_s,$ the connected case applied to the connected components of $\mathscr{G}\times_S \Spec \Og_{S,s}$ shows that the morphism $\mathscr{G}\times_S \Spec \Og_{S,s}\to \mathscr{N}_s$ is an étale monomorphism, and the second condition from the proposition shows that it is surjective. Now \cite[Tag 025G]{Stacks} shows that the morphism is a surjective open immersion, and hence an isomorphism. 
\end{proof}

Given a scheme $X$ locally of finite presentation over a field $\kappa,$ there exists a largest geometrically reduced closed subscheme $X^\natural \subseteq X;$ see \cite[Lemma C.4.1]{CGP}. If $X$ is a group scheme over $\kappa,$ then so is $X ^ \natural$ because the functor $(-)^\natural$ commutes with finite products (\it ibid.\rm). In this case, $X^\natural$ is automatically smooth over $\kappa.$ Recall that an algebraic group $N$ over $\kappa$ is said to be \it totally non-smooth \rm if $N(\kappa\sep)=0$ or, equivalently, if $N^\natural =0.$ We shall need to control the behaviour of Néron lft-models with respect to quotients with totally non-smooth kernels. 

Let $R$ be an algebra over a field $\kappa$ satisfying $[\kappa:\kappa^p] < \infty.$ For a $\kappa$-algebra $A$, let $A^{(p)}$ be the $\kappa$-algebra $A\otimes_{\kappa, F}\kappa;$ the notation here indicates that $\kappa$ is regarded as a $\kappa$-algebra via the Frobenius morphism. In particular, we have the relative Frobenius $A^{(p)} \to A$ given by $a\otimes \lambda \mapsto \lambda a^p.$ The functor on the category of $\kappa$-algebras given by $A\mapsto \Hom_\kappa(R, A^{(p)})$ is representable by a $\kappa$-algebra $R^{(1/p)}$ \cite[Chapter 7.6, first part of the proof of Theorem 4]{BLR}, which is of finite type over $\kappa$ as soon as $R$ is. Iterating this process yields the $\kappa$-algebras $R^{(1/p^n)}$ for $n\in \N.$ We have a canonical adjunction morphism $R \to (R^{(1/p)})^{(p)},$ composing which with the relative Frobenius of $R^{(1/p)}$ yields the canonical map
$$g^\ast_{R/\kappa} \colon R \to R^{(1/p)}.$$  We shall put $$R^{(1/p^\infty)} := \varinjlim R^{(1/p^n)},$$ where the limit is taken along the maps $g^\ast_{R^{(1/p^n)}/\kappa}.$\footnote{Of course, we have $(\Spec R)^{\RP}= \Spec R^{(1/p^\infty)}.$} 
\begin{proposition}
For any $\kappa$-algebra $R,$ the algebra $R^{(1/p^\infty)}$ is geometrically reduced. \label{gmredprop}
\end{proposition}
\begin{proof}
This follows from \cite[Remark 2.2 (3)]{BS} once we show that the relative Frobenius $(R^{(1/p^\infty)})^{(p)} \to R^{(1/p^\infty)}$ is an isomorphism. But this is true by construction since the adjunction maps $R^{(1/p^n)} \to (R^{(1/p^{n+1})})^{(p)}$ induce an inverse of the relative Frobenius on the limit. 
\end{proof}
\begin{proposition}
Let $N$ be an affine commutative group scheme of finite type over a field $\kappa$ satisfying $[\kappa:\kappa^p] < \infty.$ Suppose that $N$ is totally non-smooth, i. e. that $N^\natural=0.$ Then there exists some $n\gg 0$ such that the morphism $N^{(1/p^n)} \to N$ (defined as the composition $g_{N/\kappa} \circ ... \circ g_{N^{(1/p^{n-1})}/\kappa}$) vanishes. \label{vanishprop}
\end{proposition}
\begin{proof}
Write $N=\Spec R$ and let $I\subseteq R$ be the kernel of the morphism $e ^ \ast \colon R\to \kappa$ dual to the inclusion of the neutral element. By assumption, every morphism $R\to R'$ of $\kappa$-algebras with $R'$ geometrically reduced factors through $e^\ast;$ in particular, so does the canonical map $R \to R^{(1/p^\infty)}$ by Proposition \ref{gmredprop}. Since $R$ is Noetherian, $I$ is generated by elements $\alpha_1,..., \alpha_d \in R.$ By assumption, they all map to 0 in $R^{(1/p^\infty)},$ and hence must already map to 0 in $R^{(1/p^n)}$ for some large $n\in \N.$ Then $N^{(1/p^n)} \to N$ factors through the inclusion of the neutral element of $N$ as claimed. 
\end{proof}\\
$\mathbf{Remark}.$ In fact, one can show that a totally non-smooth commutative algebraic group $N$ over a field $\kappa$ is automatically affine. Indeed, by \cite[III, Théorème 8.2, Corollaire 8.3]{DG}, we can write $N$ as an extension $0\to Z \to N \to N^{\mathrm{aff}} \to 0,$ where $N^{\mathrm{aff}}$ and $Z$ are an affine and a smooth connected algebraic group over $\kappa,$ respectively. Because $N$ is totally non-smooth, we must have $Z=0,$ so $N=N^{\mathrm{aff}}.$ We shall only apply Proposition \ref{vanishprop} to algebraic groups already known to be affine.

\begin{lemma}
Let $G_1\to G_2$ be a (necessarily surjective) morphism of smooth commutative algebraic groups over $K$ such that, for all étale $K$-algebras $L,$ the induced morphism $G_1(L) \to G_2(L)$ is surjective. Suppose that $G_1$ admits a Néron (lft-)model over $S$ and that $G_2$ admits Néron lft-models over $\Og_{S,s}$ for all closed points $s\in S.$ Then $G_2$ admits a Néron (lft-)model over $S.$  \label{locglobsurjlem}
\end{lemma}
\begin{proof} 
Let $J$ be the kernel of $G_1\to G_2$ and let $\mathscr{G}_1$ be the Néron lft-model of $G_1$ over $S.$ Let $\mathscr{J}$ be the scheme-theoretic closure of $J$ in $\mathscr{G}_1;$ this scheme is flat over $S$ because $S$ is a Dedekind scheme. We claim that $\mathscr{G}_2 := \mathscr{G}_1/\mathscr{J}$ is the Néron (lft-)model of $G_2$ over $S.$ Note that this quotient is representable by \cite[Théorème 4.C]{An}. Moreover, it is smooth and separated over $S,$ and is of finite presentation over $S$ as soon as so is $\mathscr{G}_1.$ We observe that the groups of connected components of $\mathscr{G}_2$ are finitely generated for all closed points $s\in S$ because this is already true for $\mathscr{G}_1$ \cite[Proposition 3.5]{HN}. Finally, let $E\to S$ be an étale morphism of finite presentation with generic fibre $\Spec L.$ By assumption, an element $\phi\in G_2(L)$ lifts to an element of $G_1(L),$ which extends to a section in $\mathscr{G}_1(E)$ by the Néron mapping property. Hence $\phi$ extends to a morphism $E \to \mathscr{G}_1 \to \mathscr{G}_2.$ The claim therefore follows from Proposition \ref{IsNmprop}.
\end{proof}

\begin{proposition} \label{Nnonsmprop}
Assume that $r<\infty.$ Let $0 \to N \to G \to H \to 0$ be an exact sequence of commutative algebraic groups over $K$ such that $G$ and $H$ are smooth, whereas $N$ is totally non-smooth.
Suppose moreover that $H$ admits a Néron (lft-)model over $S.$ Then so does $G.$ 
\end{proposition}
\begin{proof} Choose some $n\gg 0$ such that the morphism $N^{(1/p^n)} \to N$ vanishes; this is possible by Proposition \ref{vanishprop}. We obtain the commutative diagram 
$$\begin{tikzcd}
0 \arrow[r] & N^{(1/p^n)} \arrow[r] \arrow[d]& G^{(1/p^n)} \arrow[r] \arrow[d]& H ^{(1/p^n)} \arrow[d] \\
0 \arrow[r] & N \arrow[r] &G \arrow[r] & H \arrow[r] & 0
\end{tikzcd}$$
with exact rows. Now put $G':=G^{(1/p^n)}/N^{(1/p^n)}.$ We have a closed immersion $G' \to H^{(1/p^n)}.$ If $\mathscr{H}$ is a Néron (lft-)model of $H$ over $S,$ then $\mathscr{H}^{(1/p^n)}$ is a Néron (lft-)model of $H^{(1/p^n)}$ over $S$ (this follows immediately from the Néron mapping property). In particular, $G'$ admits a Néron (lft-)model over $S$ by \cite[Chapter 10.1, Proposition 4]{BLR}. Moreover, because the left vertical arrow vanishes, the map $G^{(1/p^n)} \to G$ factors through $G'.$ Since the morphisms $g_{G^{(1/p^n)}/K}$ commute with étale base change for all $n\in \N$ by \cite[Proposition 2.5]{BS}, the morphism $G'(L) \to G(L)$ is surjective for all étale algebras $L$ over $K.$ Finally, we observe that $G$ admits Néron lft-models over $\Og_{S,s}$ for all closed points $s\in S.$ Indeed, by \cite[Chapter 10.2, Theorem 2]{BLR} we only need to show that $\Hom_K(\Ga, G)=0.$ However, $\Hom_K(\Ga, N)=0$ because $N$ is totally non-smooth over $K,$ and $\Hom_K(\Ga, H)=0$ by \cite[Chapter 10.1, Proposition 8]{BLR}. Therefore the proposition follows from Lemma \ref{locglobsurjlem}.
\end{proof}

\begin{proposition}
Suppose that $r=1.$ Let $G$ be a smooth commutative wound unipotent weakly permawound algebraic group over $K.$ Then $G$ admits a Néron lft-model over $S.$ \label{permawoundexprop}
\end{proposition}

\begin{proof}
Put $N_0:=\boldsymbol{\alpha}_p^{(1/p)}$ (this group is isomorphic to the kernel of $g_{\Ga/K}$). By \cite[Chapter 10.1, Proposition 4]{BLR}, we may replace $S$ by the integral closure of $S$ in some finite separable extension of $K.$ We may therefore assume that $G$ admits a filtration 
$$0 = G_0 \subseteq G_1\subseteq ... \subseteq G_d = G$$ for some $d\in \N_0$ such that $G_{j+1}/G_j$ is isomorphic to $\boldsymbol{\nu}_1(1)_K \cong \mathbf{G}_{\mathrm{m}}^{(1/p)}/\Gm$ or to $N_0$ for all $j=0,..., d-1$ by \cite[Theorem 9.5]{Ros}. In particular, $G_1$ is isomorphic to one of those groups. We shall now prove the claim by induction on $d.$ If $d=1,$ we must have $G\cong  \mathbf{G}_{\mathrm{m}}^{(1/p)}/\Gm$ since $G$ is smooth, in which case the claim follows from \cite[Lemma 2.9]{OvGR} (cf. also \cite[Chapter 10.1, Example 11]{BLR}). In general, $G/G_1$ admits a similar filtration of length $d-1,$ so we may assume that it admits a Néron lft-model. If $G_1\cong N_0,$ then $G$ admits a Néron lft-model by Proposition \ref{Nnonsmprop}; if $G_1 \cong \boldsymbol{\nu}_1(1)_K,$ then $G$ admits a Néron lft-model by dévissage.
\end{proof}

The tools so far assembled allow us to settle the unipotent case:

\begin{proposition}
Suppose that $r=1.$ Let $G$ be a smooth connected commutative wound unipotent algebraic group over $K.$ Then $G$ admits a Néron lft-model over $S.$ \label{woundexistprop}
\end{proposition}
\begin{proof}
By \cite[Chapter 10.2, Lemma 12]{BLR}, we can find a filtration $0=G_0 \subseteq G_1 \subseteq ... \subseteq G_d = G$ whose successive quotients are smooth, connected, wound unipotent, and annihilated by $p.$ For each $j=0,..., d-1,$  there exists a smooth wound unipotent weakly permawound algebraic group $P_j$ over $K$ and a closed immersion $G_{j+1}/G_j \to P_j$ by \cite[Theorem 1.4]{Ros}. By Proposition \ref{permawoundexprop} and \cite[Chapter 10.1, Proposition 4]{BLR}, all the $G_{j+1}/G_j$ admit Néron lft-models over $S;$ hence so does $G$ by dévissage.  
\end{proof}

We shall now give the \it proof of Theorem \ref{ExistenceThm}. \rm Recall that $S$ is an excellent Dedekind scheme over $\F_p$ whose field of rational functions $K$ satisfies $[K:K^p]=p.$ Let $G$ be a smooth commutative algebraic group over $K$ such that $\Hom_K(\Ga, G)=0.$ Let $T$ be the maximal torus of $G$ \cite[Exposé XIV, Théorème 1.1]{SGAIII2}. If we can show that $G/T$ admits a Néron lft-model over $S,$ then the Theorem follows by dévissage together with \cite[Chapter 10.1, Proposition 6]{BLR}. Note that, since $\Ext^1_K(\Ga, T)=0$ \cite[Exposé XVII, Théorème 6.1.1 A) ii)]{SGAIII2}, we have $\Hom_K(\Ga, G/T)=0$. In particular, we may assume that $G$ does not contain a non-trivial torus. By \cite[Chapter 9.1, Theorem 1]{BLR}, there exists an affine algebraic group $L$ over $K$ and an exact sequence $0 \to L \to G \to A \to 0,$ where $A$ is an Abelian variety over $K.$ Because $G$ contains neither a copy of $\Ga$ nor a torus, the identity component $L^{\natural,0}$ of $L^\natural$ is wound unipotent (cf. \cite[Proposition A.2.11]{CGP}). Therefore $L^\natural$ admits a Néron lft-model by Propositions \ref{etaleprop} and \ref{woundexistprop}. Hence it suffices to show that $G/L^\natural$ admits a Néron lft-model by dévissage. However, we have an exact sequence
$$0 \to L/L^\natural \to G/L^\natural \to A \to 0,$$ and $L/L^\natural$ is totally non-smooth over $K$ by the maximality of $L^\natural$ among smooth closed subgroups of $L.$ Since Abelian varieties over $K$ admit Néron models over $S$ \cite[Chapter 1.4, Theorem 3]{BLR}, this follows from Proposition \ref{Nnonsmprop}. \qed \\

As an application, we shall now give an elementary proof of the existence of global Néron models for pseudo-Abelian varieties. Recall that a \it pseudo-Abelian variety \rm \cite[Definition 0.1, Theorem 2.1]{Totaro} over a field $\kappa$ is a smooth connected commutative algebraic group $P$ over $\kappa$ such that the largest smooth connected affine closed subgroup of $P$ is trivial. 
\begin{theorem}
Suppose that $r<\infty.$ Let $P$ be a smooth commutative algebraic $K$-group such that the maximal smooth connected affine subgroup $U\subseteq P$ admits a Néron (lft-)model over $S.$ Then so does $P.$ In particular, every pseudo-Abelian variety over $K$ admits a Néron model over $S.$ 
\end{theorem}
\begin{proof}
Let $L$ be a closed affine subgroup of $P$ such that $P/L$ is an Abelian variety \cite[Chapter 9.1, Theorem 1]{BLR}. Then $U\subseteq L^\natural,$ and the quotient $L^\natural / U$ is étale over $K.$ By dévissage, Proposition \ref{Nnonsmprop}, and the existence of global Néron models of Abelian varieties \cite[Chapter 1.4, Theorem 3]{BLR}, $P/U$ admits a Néron model over $S.$ By dévissage and \cite[Chapter 7.2, Proposition 4, Chapter 10.1, Proposition 4]{BLR}, $P$ admits a Néron (lft-)model if and only if so does $U,$ which proves the first claim. If $P$ is pseudo-Abelian, we have $U=0$ by definition, so the second claim also follows. 
\end{proof}

We conclude by giving an arithmetic application of the existence of Néron (lft-)models:
\begin{theorem}
Let $S$ be a (not necessarily excellent) Dedekind scheme with field of rational functions $K.$ For a closed point $s\in S,$ let $K_s^{\mathrm{sh}}$ be a maximal extension of $K$ unramified at $s$ with completion $\widehat{K}_s^{\mathrm{sh}}.$ Let $f\colon G\to H$ be a smooth surjective morphism of commutative algebraic groups over $K,$ and assume that $H$ admits a Néron model over $S.$ Then, for all but finitely many closed points $s$ of $S,$ the morphisms $f(K_s^{\mathrm{sh}})$ and $f(\widehat{K}_s^{\mathrm{sh}})$ are surjective. 
\end{theorem}
\begin{proof}
Let $\mathscr{H}$ be the Néron model $H$ over $S.$ After shrinking $S$ if necessary, there exists a smooth separated model $\mathscr{G}\to S$ of $G$ of finite type. Therefore, by passing to the limit, there exists a dense open subset $U\subseteq S$ such that the induced morphism $\mathscr{G} \to \mathscr{H}$ is smooth and surjective. For a closed point $s\in S,$ let $\Og_{S,s}^{\mathrm{sh}}$ be the strict Henselization (with respect to a choice of separable closure of $\kappa(s)$) of $\Og_{S,s}.$ Then the morphism $\mathscr{G}(\Og_{S,s}^{\mathrm{sh}}) \to \mathscr{H}(\Og_{S,s}^{\mathrm{sh}})$ is surjective as soon as $s\in U.$ Since $\mathscr{H}(\Og_{S,s}^{\mathrm{sh}})=H(K_s^{\mathrm{sh}})$ by the Néron mapping property, the claim follows. The proof for $f(\widehat{K}_s^{\mathrm{sh}})$ is analogous.
\end{proof}

\noindent $\mathbf{Remark.}$ If we only required the existence of a Néron lft-model of $H,$ the conclusion would fail, as shown by the map $[n] \colon \Gm \to \Gm$ for an integer $n$ invertible on $S.$


\begin{thebibliography}{10}

\bibitem{An}
Anantharaman, S. 
\textit{Schémas en groupes, espaces homogènes et espaces algébriques sur une base de dimension 1}. Bull. Soc. Math. France, Mémoire 33 (1973), pp. 5-79.

\bibitem{BS}
Bertapelle, A., Suzuki, T. 
\textit{The relatively perfect Greenberg transform}. To appear in Manuscripta math. Available at \url{https://arxiv.org/pdf/2009.05084}.

\bibitem{BLR}
Bosch, S., L\"utkebohmert, W., Raynaud, M.
\textit{Néron models}. Ergeb. Math. Grenzgeb., Springer-Verlag, Berlin, Heidelberg, 1990.

\bibitem{CGP}
Conrad, B., Gabber, O., Prasad, G.
\textit{Pseudo-reductive Groups}. 2nd ed. New Mathematical monographs 26, Cambridge University Press, 2015.

\bibitem{DG}
Demazure, M., Gabriel, P.
\textit{Groupes algébriques. Tome I: Géométrie algébrique, généralités, groupes commutatifs}, avec une appendice `Corps de classes local' par M. Hazewinkel. Masson \& Cie, Éditeur, Paris; North-Holland Publishing Company, Amsterdam, 1970. 

\bibitem{SGAIII2}
Demazure, M.; Grothendieck, A. eds. 
\textit{Séminaire de Géométrie Algébrique du Bois Marie - 1962-64 - Schémas en groupes II}. Lecture Notes in Math. 152, Springer-Verlag, 1970.  

\bibitem{Dix}
Grothendieck, A.
\textit{Le groupe de Brauer III: Exemples et complements}, In \it Dix exposes sur la cohomologie des schémas. \rm Advanced Studies in Pure Mathematics, A. Grothendieck, N. H. Kuiper eds. Masson \& Cie, Éditeur, Paris; North-Holland Publishing Company, Amsterdam, 1968. 

\bibitem{HN} Halle, L. H., Nicaise, J.
\textit{Motivic zeta functions of abelian varieties, and the monodromy conjecture}. Adv. Math., 227(1), pp. 610-653, 2011.  

\bibitem{Ill} Illusie, L.
\textit{Complexe de de Rham-Witt et cohomologie cristalline}. Ann. sci. École Norm. Sup., Serie 4, Vol. 12 no. 4, pp. 501-661, 1979.  

\bibitem{Kato} Kato, K.
\textit{Duality Theories for the $p$-Primary Etale Cohomology. I}. Algebraic and Topological Theories - to the memory of Dr. Takehiko MIYATA, pp. 127-148, 1985.

\bibitem{LLR}
Liu, Q., Lorenzini, D., Raynaud, M., 
\textit{Néron models, Lie algebras, and reduction of curves of genus one}. Invent. math. 157, pp. 455-518, 2004.

\bibitem{Milne}
Milne, J. S.
\textit{Étale Cohomology}. Princeton University Press, Princeton, 1980.

\bibitem{Oes} Oesterlé, J. 
\textit{Nombres de Tamagawa et groupes unipotents en caractéristique p}. Invent. math. 78, pp. 13-88, 1984.

\bibitem{OvChai} Overkamp, O.
\textit{Chai's conjecture for semiabelian Jacobians}. Compositio Math. 161 (2025), pp. 120–147.

\bibitem{OvGR} Overkamp, O.
\textit{On Jacobians of geometrically reduced curves and their Néron models}. Trans. Amer. Math. Soc., Vol. 377, Nr. 8, pp. 5863-5903, 2024. 

\bibitem{OS} Overkamp, O., Suzuki, T. 
\textit{Existence of global Néron models beyond semi-abelian varieties.} Preprint; Available at \url{https://arxiv.org/pdf/2310.14567}.

\bibitem{Ros} Rosengarten, Z.
\textit{Permawound Unipotent Groups}. Transformation Groups, 2024. 

\bibitem{Ros2} Rosengarten, Z. 
\textit{Rigidity and unirational groups}. Preprint; available at \url{https://arxiv.org/abs/2307.04649}.

\bibitem{Shi}
Shiho, A., 
\textit{On Logarithmic Hodge-Witt Cohomology of Regular Schemes}. J. Math. Sci. Univ. Tokyo 14, pp. 567–635, 2007.

\bibitem{Suz} Suzuki, T.
\textit{Class field theory, Hasse principles and Picard-Brauer duality for two-dimensional local rings}. To appear in Algebraic Geometry. Available at \url{https://arxiv.org/pdf/2210.01396}.

\bibitem{Stacks}
Authors of the Stacks project.
\textit{Stacks project}. Columbia University.

\bibitem{Totaro}
Totaro, B.
\textit{Pseudo-Abelian varieties}. Ann. sci. École Norm. Sup., Serie 4, Vol. 46, no. 5, pp. 693-721, 2013.

\end{thebibliography}
\end{document}